\documentclass[12pt]{amsart}

\newcounter{defcounter}
\usepackage{amssymb,amsmath,amscd,graphicx, enumitem,
latexsym,amsthm}
\usepackage{amssymb,latexsym,eufrak,amsmath,amscd,graphicx}
  \usepackage[all]{xy}
  \usepackage{pdfsync}
\theoremstyle{plain}
\newtheorem{theorem}{Theorem}
\newtheorem{proposition}[theorem]{Proposition}
\newtheorem{corollary}[theorem]{Corollary}
\newtheorem{lemma}[theorem]{Lemma}

\newtheorem{proposition.definition}[theorem]{Proposition/Definition}

\newtheorem{theoremalpha}{Theorem}
\newtheorem{corollaryalpha}[theoremalpha]{Corollary}

\newtheorem{conjecture}[theorem]{Conjecture}

\theoremstyle{definition}
\newtheorem{definition}[theorem]{Definition}
\newtheorem{setup}[theorem]{Set-Up}

\newtheorem{remark}[theorem]{Remark}
\newtheorem{example}[theorem]{Example}

\newtheorem{problem}[theorem]{Problem}

\newcommand{\lra}{\longrightarrow}

\newcommand{\noi}{\noindent}
\newcommand{\PP}{\mathbf{P}}

\newcommand{\CC}{\mathbf{C}}
\newcommand{\QQ}{\mathbf{Q}}

\newcommand{\OO}{\mathcal{O}}

\newcommand{\II}{\mathcal{I}}

\newcommand{\HH}[3]{H^{{#1}} \big( {#2} , {#3}
\big) }

\newcommand{\ord}{\textnormal{ord}}
\newcommand{\codim}{\textnormal{codim}}

\newcommand{\pr}{\prime}

\newcommand{\lin}{\equiv_{\text{lin}}}

\newcommand{\dra}{\dashrightarrow}
\newcommand{\Bl}{\text{Bl}}

\newcommand{\Linser}[1]{| \mspace{1.5mu} {#1}
\mspace{1.5mu} |}
\newcommand{\linser}[1]{\Linser{  {#1}  }}

\newcommand{\gon}{\textnormal{gon}}
\newcommand{\pro}{{pr}}

\newcommand{\Sym}{\textnormal{Sym}}

\newcommand{\GG}{\mathbf{G}}
\newcommand{\chow}{\textnormal{CH}}
\newcommand{\trace}{\textnormal{Tr}}

\newcommand{\irrdeg}{\textnormal{irr}}
\newcommand{\conngon}{\textnormal{conn.\,gon}}
\newcommand{\covgon}{\textnormal{cov.\,gon}}
\newcommand{\BVA}[1]{\textnormal{(BVA)}_{#1}}

\newcommand{\Sec}{\textnormal{Sec}}

\numberwithin{theorem}{section}

\begin{document}

\title[Measures of irrationality]
{Measures of Irrationality for Hypersurfaces of Large Degree
 }

\author[Bastianelli]{Francesco Bastianelli}
\address{Dipartimento di Matematica, Universit\`a degli Studi di Bari, Via Edoardo Orabona 4, 
70125 Bari}
\email{\tt 
francesco.bastianelli@uniba.it}
\thanks{Research of the first author partially supported by FIRB 2012 ``Spazi di moduli e applicazioni''.}

\author[De Poi]{Pietro De Poi}
\address{Dipartimento di Scienze Matematiche, Informatiche e Fisiche, Universit\`a degli Studi di Udine, Via delle Scienze 206, 33100 Udine}
\email{\tt pietro.depoi@uniud.it}
\thanks{Research of the second author partially supported by MIUR PRIN 2010--2011 ``Geometria delle variet\`a algebriche''.}

 \author[Ein]{Lawrence Ein}
  \address{Department of Mathematics, University Illinois at Chicago, 851 South Morgan St., Chicago, IL  60607}
 \email{{\tt ein@uic.edu}}
 \thanks{Research of the third author partially supported by NSF grant DMS-1501085.}

 \author[Lazarsfeld]{Robert Lazarsfeld}
  \address{Department of Mathematics, Stony Brook University, Stony Brook, New York 11794}
 \email{{\tt robert.lazarsfeld@stonybrook.edu}}
 \thanks{Research of the fourth author partially supported by NSF grant DMS-1439285.}

 \author[Ullery]{Brooke Ullery  
  } \address{Department of Mathematics, Harvard University, Cambridge, MA 02138}
 \email{{\tt ullery@math.harvard.edu}}
 \thanks{Research of fifth author partially supported by an NSF Postdoctoral Fellowship,    DMS-1502687}

 \dedicatory{Dedicated to J\'anos Koll\'ar on the occasion of his sixtieth birthday}

\maketitle
\setlength{\parskip}{.13in minus .03in}

     \section*{Introduction}
 
 There has been a great deal of recent interest and progress in studying  issues of  rationality  for algebraic varieties (cf. \cite{Kollar}, \cite{Voisin2}, \cite{VoisinChow2}, \cite{Totaro}, \cite{HPT}). The purpose of this paper is to investigate a complementary circle of questions:  in what manner can one quantify and control ``how irrational" a given projective variety $X$ might be? We consider various measures of irrationality for hypersurfaces of large degree in projective space and other varieties. The theme is that  positivity properties of  canonical bundles lead to lower bounds for these invariants. In particular, we  prove the  main conjecture of  \cite{BCD}  that if $X \subseteq \PP^{n+1}$ is a very general smooth hypersurface of dimension $n$ and degree $d \ge 2n+1$ then any dominant rational mapping $f : X \dra \PP^n$ must satisfy
\[  \deg(f) \, \ge \, d-1. \]
We also propose a number of open problems involving this circle of ideas, and we show how our methods lead to simple new proofs of theorems of Ran \cite{Ran} and Behesti--Eisenbud \cite{BE} concerning varieties of multi-secant lines.

To start with some background,  recall that the \textit{gonality} $\gon(C)$ of  an irreducible complex projective curve $C$ is defined to be the least degree of a branched covering \[ C^\pr \lra \PP^1, \]where $C^\pr$ is the normalization of $C$.  Thus 
\[ \gon(C) \ = \ 1 \ \ \Longleftrightarrow \  \ C \, \approx_{\text{birat}} \, \PP^1, \] and it is profitable in general to view $\gon(C)$  as measuring the failure of $C$ to be rational. Because of this there has  been a certain amount of interest over the years in  bounding from below the gonality of various natural classes of curves. For instance, a classical theorem of Noether states that if $C \subseteq  \PP^2$ is a smooth plane curve of degree $d \ge 3$, then \[ \gon(C) \ = \  d-1, \] with the relevant coverings given by projection from a point of $C$. This was generalized to complete intersection and other curves in \cite[Ex. 4.12]{LLS} and \cite{HS} by means of  vector bundle techniques.  Abramovich \cite{Abramovich} used results of Li and Yau to obtain a linear lower bound on the gonality of modular curves. In a somewhat different direction, it was established in \cite{Lengths} that the Buser-Sarnak invariant of $\text{Jac}(C)$ is linearly bounded above  by $\gon(C)$, and  the behavior of gonality in certain towers of coverings was studied by Hwang and To \cite{HwangTo} as a consequence of relations they established between gonality and injectivity radii. The paper \cite{BT} contains some interesting applications of the results of Hwang--To.

Several authors have  proposed and studied some analogous measures of irrationality for an irreducible complex projective variety $X$ of arbitrary dimension $n$. We will be principally  concerned here with  three of these  -- the \textit{degree of irrationality}, the \textit{connecting gonality}, and the \textit{covering gonality} of $X$ -- defined as follows:
\begin{align*} \irrdeg(X) \ &= \ \min \Big  \{ \delta > 0 \  
\Big | \parbox{2.2in}{\begin{center} $\exists$ degree $\delta$  rational covering \\$X \dra \PP^n$  \end{center}}
\Big  \};  \\ \\
\conngon(X) \ &= \ \min \Bigg  \{ c > 0 \  
\Bigg | \parbox{2.7in}{\begin{center} General points $x, y \in X$ can be connected by an irreducible curve $C \subseteq X$ with  $\gon(C) = c$. \end{center}}
\Bigg \}; 
\\ \\
\covgon(X) \ &= \ \min \Bigg  \{ c > 0 \ 
\Bigg | \parbox{2.7in}{\begin{center} Given a general point $x \in X$, $\exists$ an irreducible curve $C \subseteq X$ through $x$ with  $\gon(C) = c$. \end{center}}
\Bigg  \}. \end{align*}
(Note that the curves $C$ computing the connecting and covering gonalities  are allowed to be singular.) Thus
\begin{align*}
\irrdeg(X) \, = \, 1 \ &\Longleftrightarrow \ X \text{ is rational},\\
\conngon(X) \, = \, 1 \ &\Longleftrightarrow \ X \text{ is rationally connected},\\
\covgon(X) \, = \, 1 \ &\Longleftrightarrow \ X \text{ is uniruled},\end{align*}
and in general  one has the inequalities
\begin{equation} \label{relations.among.invariants}  \covgon(X) \, \le \, \conngon(X) \, \le \, \irrdeg(X). \end{equation}
The  integer $\irrdeg(X)$ is   perhaps the most natural generalization  of the gonality of a curve, but $\covgon(X)$ often seems to be easier to control.\footnote{We introduce the connecting gonality only because it fits naturally into the  picture. In fact this invariant does not enter seriously into any of our results.} Another invariant, suggested by Voisin, is the least degree $v(X)$ of one-parameter families of Chow-constant zero cycles that cover $X$. This satisfies $v(X) \le \covgon(X)$. 

The degree of irrationality was introduced by Heinzer and Moh in \cite{HM}, and Yoshihara subsequently computed it for several classes of surfaces   (\cite{Y1}, \cite{Y2}, \cite{Y2.5}, \cite{Y3}, \cite{Y4}). Lopez and Pirola \cite{LP} showed in passing that if $X \subseteq \PP^3$ is a surface of degree $d \ge 4$, then $\covgon(X) = d-2$.  Along similar lines,  Fakhruddin  established in his note \cite{F} that given any integer $c > 0$, a very general hypersurface of sufficiently large degree in any smooth variety does not contain any curves of gonality $\le c$. However as a measure of irrationality, it seems that the covering gonality was first studied systematically in \cite{Bast}, where the first author computes $\covgon(X)$ and bounds $\irrdeg(X)$ when $X = C_2$  is the symmetric square of a curve $C$. 

The present work was most directly motivated by the   paper \cite{BCD} in which Cortini and the first two authors consider the question of computing the degree of irrationality of a smooth projective hypersurface
\[    X \, = \, X_d \ \subset \ \PP^{n+1}  \]
of degree $d$ and dimension $n \ge 2$, generalizing the result of Noether for plane curves cited above. They show to begin with that if $d \ge n+3$ then 
\begin{equation} \label{BCP.Bound}
d -n   \ \le \ \irrdeg(X) \ \le d -1. 
\end{equation}
It can happen that $\irrdeg(X) < d-1$, but it was established in \cite{BCD} that if $X$ is a \textit{very general} surface of degree $d \ge 5$ or threefold of degree $d \ge 7$ then
\[  \irrdeg(X) \ = \ d-1, \]
and  the exceptional cases were classified in these dimensions. It was conjectured there that this statement extends to hypersurfaces of all dimensions.

Our first results concern covering gonality.
\begin{theoremalpha} \label{Cov.Gon.Thm}
Let $X \subseteq \PP^{n+1}$ be a smooth hypersurface of dimension $n$ and degree $d \ge n+2$. Then
\[  \covgon(X) \ \ge \ d - n . \]
\end{theoremalpha}
\noi More generally, we show that $v(X) \ge d-n$. Observe that one recovers in particular the lower bound \eqref{BCP.Bound} of Bastianelli, Cortini and De Poi on the degree of irrationality of such hypersurfaces. In fact  it suffices in the Theorem that $X$ is normal with at worst canonical singularities, and in this setting  the statement is best possible for every $n \ge 2$ and $d \ge n+2$.   We actually prove that the conclusion of the Theorem and the bound for $v(X)$ hold for any smooth projective variety $X$ with
\[   K_X \ \lin \ B + E \]
where $B$ is a $(d-n-2)$-very ample divisor on $X$  and $E$ is effective.\footnote{Recall that a divisor $B$ on a smooth projective variety $Y$ is said to be $p$-very ample if any finite subscheme $\xi \subseteq Y$ of length $(p+1)$ imposes independent conditions on $H^0(Y,B)$. If $A$ is a very ample divisor then $\OO_Y(pA)$ is $p$-very ample, and therefore if  $X \subseteq \PP^{n+1}$ is a smooth hypersurface of degree $d$ then $K_X$ is $(d-n-2)$-very ample.} Thus  we deduce
\begin{corollaryalpha}
Let $M$ be a smooth projective variety, and let $A$ be a very ample divisor on $M$. There is an integer $e = e(M, A)$ depending only on $M$ and $A$ with the property that if
\[   X_d \, \in \ \linser{dA} \]
is any smooth divisor, then 
\[  \covgon(X_d) \ \ge \ d - e. \]
\end{corollaryalpha}
\noi In particular, the degree of irrationality of $X_d$  goes to infinity with $d$. (One can prove this last fact directly using the ideas of \cite{LP}, \cite{Bast},  \cite{BCD}, and \cite{GP}: see Remark \ref{Direct.Proof.Irr.Bound}.)  As noted above, Fakhruddin proved in \cite{F} the closely related result that in the situation of the Corollary, there is a linear function $d(c)$ such that a very general divisor $X_d \in \linser{dA}$ actually contains no curves of gonality $\le c$ provided that $d \ge d(c)$. (Compare Proposition \ref{No.Curves.Small.Gonality} below.)

Returning to smooth hypersurfaces in projective space, our second theorem proves the main conjecture of \cite{BCD}: \begin{theoremalpha} \label{BCP.Conj.Thm.}
Let $X \subseteq \PP^{n+1}$ be a very general smooth hypersurface of dimension $n$ and degree $d \ge 2n+1$. Then
\[  \irrdeg(X) \, = \, d-1. \]
Furthermore, if $d \ge 2n+2$ then any rational mapping 
\[  f : X \dra \PP^n \ \ \text{ with } \deg(f) = d-1 \]
is given by projection from a point of $X$.
\end{theoremalpha}

The proof of Theorem \ref{Cov.Gon.Thm}, which is quite quick and elementary, occupies \S1: there we work on an arbitrary smooth variety whose canonical bundle satisfies a suitable positivity property.  Voisin's invariant $v(X)$ is studied in \S 2 using the ideas introduced by Mumford in \cite{Mumford}. These results actually imply Theorem \ref{Cov.Gon.via.BVA.Thm}, but we felt that it was worthwhile nonetheless to present the elementary and transparent proof of that statement.  

For Theorem \ref{BCP.Conj.Thm.}, which appears in \S 3, we start with  the set-up established in  \cite{BCD}. It is shown there
\[  f : X \dra \PP^n \]
is a rational mapping with \[ d-n \   \le \   \deg(f) \ \le \  d-2, \] then each of  the fibres of $f$ spans a line in $\PP^{n+1}$ provided that $d \ge 2n+1$. Furthermore  these lines form a congruence  of order one on $\PP^{n+1}$, meaning that a general point of $\PP^{n+1}$ lies on exactly one of the lines.  The main effort in \cite{BCD} was to use results on the classification of such congruences to show that $X$ must contain a rational curve when $n = 2$ or  $n = 3$, which forces $X$ to be special provided that $d \ge 2n+1$. The  new point here is the observation   that in arbitrary dimension $n$, whether or not $X$  contains a rational curve,    one can  locate on $X$ a relatively large subvariety covered by  curves of gonality $e \le n$. On the other hand, drawing on   computations of the third author and Voisin  in \cite{Ein} and \cite{Voisin},  one can bound the dimension of a subvariety of small covering gonality in a very general hypersurface. Theorem \ref{BCP.Conj.Thm.} follows. The common thread in these arguments is that the  invariants   we consider are ultimately controlled by measuring the positivity of canonical bundles. 

In \S4 we  present a number of conjectures and open problems. It turns out that the computations in \S 3 also lead to  quick new proofs of results of Ran \cite{Ran} and Beheshti--Eisenbud \cite{BE} concerning varieties with many highly secant lines. These appear in the Appendix.

The reader will see that the methods of the present paper are rather elementary, and several of the ideas involved are at least  implicit in earlier work such as \cite{LP}, \cite{F}, \cite{Kn}, \cite{Bast}, \cite{BCD} and \cite{GP}. However we have tried to pull things together  in a natural way by focusing on a specific birational measure of positivity for the canonical bundle (Definition \ref{BVA}). We hope that this might help to lay the foundation for further work on    what we consider to be an interesting circle of questions. 

In an earlier version of this paper, the  last three authors proved the conjecture of \cite{BCD} under the stronger numerical hypothesis $d \ge 3n$, which the first two authors weakened somewhat in an appendix. Voisin subsequently showed us how to get this down to $d > \frac{5}{2}n$, after which the first two authors were able to prove Theorem \ref{BCP.Conj.Thm.} as stated above.\footnote{We note that David Yang independently gave  the improvement to $d > \tfrac{5}{2}n$ of our original bound $d \ge 3n$.} The present paper represents a pooling of these efforts.

We are grateful to Ron Donagi, Daniel Litt, Luigi Lombardi, John Ottem, Ian Shipman, David Stapleton,  Jason Starr and Damiano Testa for helpful discussions.  We would also like to acknowledge the influence of Pietro Pirola, who with his colleagues introduced many of the ideas that implicitly play a role here. We are particularly grateful to Claire Voisin, who besides suggesting the numerical improvement just noted proposed the material that appears in \S 2. We are also grateful to another referee of an earlier version for providing  several  valuable expository and mathematical suggestions. 

We are honored to dedicate this paper to J\'anos Koll\'ar on the   occasion of his sixtieth birthday. Beyond guiding the direction of algebraic geometry over three decades, J\'anos has  been instrumental to the work of the third and fourth authors through his encouragement and generosity with ideas. It is a pleasure to have this opportunity to express our admiration and thanks.

Concerning notation and conventions -- we work throughout over the complex numbers. As usual, morphisms are indicated by solid arrows, while rational mappings are dashed. We have taken the customary liberties in confounding line bundles and divisors. 

\numberwithin{equation}{section}

\section{Birational Positivity and Covering Gonality}

In this section we study the covering gonality of a projective variety $X$, and prove Theorem \ref{Cov.Gon.Thm}
 from the Introduction. The basic strategy is to  bound $\covgon(X)$ in terms of the positivity of the canonical bundle $K_X$. So we start with some remarks on birational measures of positivity  for line bundles.

 Let $X$ be an irreducible projective variety. Given an integer $p \ge 0$, recall that a line bundle $L$  on $X$ is said to be $p$-\textit{very ample} if the restriction map
\[
\HH{0}{X}{L} \lra \HH{0}{X}{L \otimes \OO_{\xi}} 
\] is surjective for every finite subscheme $\xi \subseteq X$ of length $p+1$. In other words, one asks that  every subscheme of length $p+1$ imposes independent  conditions on the sections of $L$. The condition we focus on here is a birational analogue of this.
\begin{definition} \label{BVA}
A line bundle $L$ on $X$ \textit{satisfies property $\BVA{p}$} if there exists a proper Zariski-closed subset $Z = Z(L) \subsetneqq X$ depending on $L$ such that
\begin{equation} \label{DefBVA} \HH{0}{X}{L} \lra \HH{0}{X}{L \otimes \OO_{\xi}} 
\end{equation} 
surjects for every finite subscheme $\xi \subset X$ of length $p+1$ whose support is disjoint from $Z$. 
\end{definition}
\noi Thus $\BVA{0}$ is equivalent to requiring that $L$  be effective, and $\BVA{1}$ is what is often called ``birationally very ample."\footnote{Hence ``BVA."} This property was considered in \cite{Kn} and \cite{KSS} (under a different name).

The following remarks yield a supply of examples.
\begin{example} \label{Examples.BVA} Let $X$ be an irreducible projective variety.
\begin{enumerate}
\item[(i).] If $L$ is a line bundle on $X$ satisfying $\BVA{p}$ and $E$ is an effective divisor on $X$, then  $\OO_X(L+E)$ satisfies $\BVA{p}$.
\vskip 4pt

\item[(ii).] Suppose that $f : X \lra Y$ is a birational morphism of irreducible projective varieties. If $L$ is a line bundle on $Y$ satisfying $\BVA{p}$, then $f^*L$ satisfies $\BVA{p}$ on $X$.
\vskip4pt
\item[(iii).] More generally,  let $f : X \lra Y$ be a morphism which is birational onto its image, and suppose that $L$ satisfies $\BVA{p}$ on $Y$. Assume moreover that $f(X)$ is not contained in the exceptional set $Z \subseteq Y$ arising in the definition of property $\BVA{}$. Then $f^*L $ satisfies $\BVA{p}$ on $X$.
\vskip 4pt
\item[(iv).] Suppose that 
\[ f : X \lra \PP \] is a morphism  from $X$ to some projective space which is birational onto its image. Then $f^*\OO_\PP(p)$ satisfies $\BVA{p}$.
\vskip4pt
\item[(v).] Suppose that 
\[  X  \ \subseteq \ \PP^{n+1} \]
is a normal hypersurface of degree $d\ge n+2$ with at worst canonical singularities, and let 
$ \mu : X^\pr \lra X$
be a resolution of singularities. Then the canonical bundle $K_{X^\pr}$ of $X^\pr$ satisfies $\BVA{d-n-2}$. 
\end{enumerate}

\noi Indeed, (i), (ii) and  (iii) are clear from the definition, while (iv) is a consequence of  (ii) and the elementary fact that $\OO_\PP(p)$ is $p$-very ample. For (v), it follows from the definition of canonical singularities that 
\[   K_{X^\pr} \ \lin \ (d - n -2)H + E, \]
where $H$ is the pullback of the hyperplane bundle on $X$ and $E$ is effective. So the assertion follows from (i) and (iv). \qed
\end{example} 

The relevance of this notion to questions of gonality arises from the following  elementary observation.
\begin{lemma} \label{Gon.Bound.Curve.Lemma}
Let $C$ be a smooth projective curve of genus $g$ whose canonical bundle $K_C$ satisfies $\BVA{p}$. Then 
\[ \gon(C) \ge p+2.\]
\end{lemma}
\begin{proof}
We may suppose $g \ge 2$. Let $A$ be a globally generated line bundle of degree $d \le g-1$ on $C$. Then the divisor $\xi$ of any section of $A$ fails to impose independent conditions on $\linser{K_C}$. Hence if $K_C$ satisfies $\BVA{p}$ then one must have $d \ge p+2$. \end{proof}

We now turn to  coverings by curves of specified gonality. Let $X$ be an irreducible projective variety. 
\begin{definition} \label{Def.Cov.Fam.Gon} A \textit{covering family of curves of  gonality $c$} on $X$ consists of a  smooth family
\[  \pi :  \mathcal{C} \lra T \]
of  irreducible projective curves parametrized by an irreducible variety $T$, together with a dominant morphism
\[   f : \mathcal{C} \lra X, \]
satisfying:
\begin{enumerate}
\item[(i).] For a general point $t \in T$, the fibre
$C_t \,=_\text{def} \, \pi^{-1}(t) $
is a smooth curve with $\gon(C_t) = c$; and
\vskip 5pt
\item[(ii).] For general $t \in T$, the  map
$ f_t : C_t \lra X$
is birational onto its image.
\end{enumerate}
\end{definition}
\noi By standard arguments, the existence of such a family is equivalent to asking that $X$ contains a (possibly singular) curve of gonality $c$ passing through a general point.

\begin{remark} \label{Cov.Gon.Properties} We make some remarks about the formal properties of this definition. 

\noi  (i). \ After replacing $T$ by a desingularization, one can suppose without loss of generality that $T$ and $\mathcal{C}$ are non-singular. 

\noi (ii). \ Given a covering family as above, after restricting to a suitable suvariety of $T$ we may suppose without loss of generality that $\dim \mathcal{C} = \dim X$, so that in particular the morphism
\[  f : \mathcal{C} \lra X \] is generically finite. 

\noi (iii). \ Suppose that 
$  \pi : \mathcal{C} \lra T$, $ f : \mathcal{C} \lra X $
is a covering family   with $\mathcal{C}$ and $T$ non-singular, and let $\nu :   \mathcal{C}^\pr  \lra \mathcal{C}$
be the blowing up of $\mathcal{C}$ along a smooth center. Then there is a non-empty Zariski-open subset $T_0 \subseteq T$ over which the restrictions of the two maps
\[   \mathcal{C}^\pr \lra T  \ \ , \ \ \mathcal{C} \lra T \]
coincide.
(Since blowing up along a divisor has no effect, we can assume that this center has codimension $\ge 2$, and hence maps to a subset of $T$ having codimension $\ge 1$.)

\noi (iv). Let  $\pi : \mathcal{C} \lra T$, $f : \mathcal{C} \lra X$ be a covering family with $\mathcal{C}$ and $T$ smooth, and let 
\[ \mu :  X^\pr \lra X \]
be a birational morphism. Then there is a non-empty Zariski-open subset $T_0\subseteq T$ so that the restriction 
$ \pi_0: \mathcal{C}_0  \lra T_0 $
extends to a family 
\[  f^\pr : \mathcal{C}_0 \lra X^\pr. \]
(In fact, by a suitable sequence of blow-ups, we can  construct a modification $\mathcal{C}^\pr \lra \mathcal{C}$ that admits an extension $f^\pr : \mathcal{C}^\pr \lra X^\pr$. The assertion then follows from (iii).)
\end{remark}

As in the Introduction, we focus on the smallest gonality of such a covering family: \begin{definition} The \textit{covering gonality} $\covgon(X)$ of $X$ is the least integer $c>0$ for which such a covering family exists. 
\end{definition}
\noi It follows from Remark \ref{Cov.Gon.Properties} (iv) that this is indeed a birational invariant. 

\begin{example} \textbf{(Examples of covering gonality).} \label{Cov.Gon.Exs}  Here are some examples where the covering gonality can be estimated or computed. 

 \noi (i). \  Let $X$ be a $K3$ surface. By a theorem of Bogomolov and Mumford (\cite[p. 351]{MM}) $X$ is covered by (singular) elliptic curves. Hence $\covgon(X) = 2$. If $X$ is an abelian surface, then similarly $\covgon(X) = 2$: in fact, $X$ is covered by curves of genus $\le 2$.\footnote{This is evident if $X$ is principally polarized, and in general $X$ is covered by a such a surface.} 
  
\noi (ii). \ Let $X = C_2$ be the symmetric square of a smooth   curve $C$ of genus $g \ge 3$. Then $X$ is covered by copies of $C$ via the double covering $C \times C \lra X$. Bastianelli \cite{Bast} shows that these curves compute the covering gonality of $X$, ie $\covgon(X) = \gon(C)$.

\noi (iii). \    Let $X \subseteq \PP^3$ be a smooth surface of degree $d \ge 4$,  let $x \in X$ be a general point, and let $T_x \subseteq \PP^3$ be the tangent plane to $X$ at $x$. Then
\[  D_x = T_x \cap X\]
is an irreducible plane curve of degree $d$ with a double point, which has gonality $d-2$. Therefore $\covgon(X) \le d-2$. In fact, Lopez and Pirola \cite{LP} show that this is the unique family of minimal gonality for general $X$, and hence $\covgon(X) = d-2$. One can evidently arrange for such a curve to pass through two general points of $X$, and hence $\conngon(X) = d-2$. 

\noi (iv). \ Suppose now that $X \subseteq \PP^4$ is a smooth threefold of degree $d \ge 5$. A dimension count predicts that  $X$ should be covered by a two-dimensional family of plane curves of degree $d$ with  triple points. One can prove -- either directly or (as Jason Starr pointed out) by a degeneration -- that this is indeed the case. Hence \[ \covgon (X) \ \le  \ d - 3, \] and the same inequality holds \textit{a fortiori} for hypersurfaces of degree $d$ and  larger dimension. It then follows from Corollary \ref{Cov.Gon.Hypsfs} that $\covgon(X) = d-3$ for a general threefold of degree $d$.

\noi (v). \ Let $X \subseteq \PP^{n+1}$ be a hypersurface of degree $d> n$ having an ordinary singular point $p \in X$ of multiplicity $n$: in particular, $X$ has only canonical singularities. Projection from $p$ gives rise to a rational map $X \dra \PP^n$ of degree $d - n$, and the inverse images of lines $\ell \subseteq \PP^n$ then yield a covering of $X$ by curves of gonality $\le d-n$. Therefore $\covgon(X) \le d-n$, and it follows from Corollary \ref{Cov.Gon.Hypsfs} below that in fact $\covgon(X) = d-n$.  \qed
 \end{example}
 
 \begin{remark} \textbf{(Covering gonality of very general hypersurface.)} By starting with  some of  the ideas going into the proof of Theorem \ref{BCP.Conj.Thm.} in \S \ref{Deg.Irrat.Section}, the first author, Ciliberto, Flamini and Supino \cite{BCFetal} have computed the covering gonality of a \textit{very general} hypersurface $X_d \subseteq \PP^{n+1}$ of degree $d \gg 0$ in almost all cases. Specifically, they show that \[   \covgon(X_d) \ \approx \ d - 2\sqrt{n}. \]
The numerics here are essentially what one find by looking for plane curves with singular points covering $X$, as in Examples \ref{Cov.Gon.Exs} (iii) and (iv). \qed
\end{remark}

Recall that if $D \lra C$ is a branched covering of irreducible projective curves, then $\gon(D) \ge \gon(C)$. This implies the analogous statement for covering gonality.   \begin{lemma} \label{Covering.Lemma}
 Let $f : X \lra Y$ be a generically finite surjective mapping between irreducible projective varieties. Then 
 \[  \covgon(X) \ge \covgon(Y). \qed\]
 \end{lemma}

The main theorem of this section asserts that the covering gonality of a smooth projective variety is bounded by the positivity of its canonical bundle. When $\dim X = 2$ the statement was established in \cite[\S 3]{KSS}. 
\begin{theorem} \label{Cov.Gon.via.BVA.Thm}
Let $X$ be a smooth projective variety, and suppose that there is an integer $p \ge 0$ such that its canonical bundle $K_X$ satisfies property $\BVA{p}$. Then
\[  \covgon(X) \ \ge \ p+2. \]
\end{theorem}
\begin{proof}
This is very elementary. Suppose that 
\[  \pi : \mathcal{C} \lra T \ \ , \ \ f : \mathcal{C} \lra X \]
is a covering family of curves of gonality $c$. Thanks to Remark \ref{Cov.Gon.Properties} (i) and (ii), there is no loss of generality in assuming that $\mathcal{C}$ and ${T}$ are smooth, and that $f$ is generically finite. Then 
\[   K_\mathcal{C} \ \lin \ f^* K_X \, + \, E  \tag{*} \]
where $E = \text{Ram}(f)$ is the ramification divisor of $f$. On the other hand, since $\pi$ is smooth  one has
\[   K_{C_t} \ \lin \ K_{\mathcal{C}} \mid C_t \tag{**} \]
for every $t \in T$. Furthermore, if $t \in T$ is general, then $C_t$ meets the effective divisor $E$ properly, and its image 
\[  f_t(C_t) \ \subseteq \ X \]
will not be contained in the exceptional set $Z(K_X) \subseteq X$ arising in Definition \ref{BVA}.   Since by definition $f_t : C_t \lra X$ is birational onto its image, if follows from (*), (**) and Example \ref{Examples.BVA} (iii) that $K_{C_t}$ satisfies property $\BVA{p}$. Hence $c \ge p+2$ thanks to Lemma \ref{Gon.Bound.Curve.Lemma}. \end{proof}

\begin{corollary} \label{Cov.Gon.Hypsfs} Let $X \subseteq \PP^{n+1}$ be a smooth hypersurface of degree $d \ge n+2$. Then \[ \covgon(X) \ \ge \ d-n. \] The same statement holds if $X$ is normal with only canonical singularities.
\end{corollary}
 \noi Note that if we allow canonical singularities, then Example \ref{Cov.Gon.Exs} (v) shows that the statement is best possible for all $n \ge 2$ and $d \ge n+2$.  When $n = 1$ we recover Noether's result that a smooth plane curve of degree $d$ has gonality $d-1$. 

\begin{proof}[Proof of Corollary]
When $X$ is smooth, its canonical bundle $\omega_X = \OO_X(d-n-2)$ is already $(d-n-2)$-very ample. For the second statement, we can pass to a desingularization, and then Example \ref{Examples.BVA} (v) applies. 
\end{proof}

We observe next that a sufficiently positive divisor on any smooth variety has large covering gonality. 
\begin{corollary}
Let $M$ be a smooth projective variety, and let $A$ be a very ample line bundle on $M$. Fix an integer $e$ such that 
 $\linser{(e+2)A + K_M}$ is basepoint-free, and let \[ X = X_d \in \linser{dA}\] be any smooth divisor. Then
 \[  \covgon(X) \ \ge \ d - e. \]
\end{corollary}
\begin{proof}
In fact, \[ K_{X} \ =  \ \big (K_M + dA \big)\mid X \ = \ \big( (d-e-2)A + E \big) \mid X, \]
where $\linser{E}$ is free. Since $A$ is very ample, $\OO_X\big((d-e-2)A\big)$ is $(d-e-2)$-very ample, and therefore $K_{X}$ satisfies Property $\BVA{d-e-2}$. 
\end{proof}

Finally, we say a word about the connecting gonality of an irreducible projective variety $X$. An evident modification of Definition \ref{Cov.Gon.Exs} leads to the notion of a family of curves of gonality $c$ connecting two general points of $X$, and as in the Introduction the least such gonality is defined to be $\conngon(X)$. Clearly  \[ \covgon(X) \ \le \ \conngon(X),\] and the example of a uniruled variety which is not rationally connected shows that the inequality can be strict. Moreover the analogue of Lemma \ref{Covering.Lemma} remains valid. Unfortunately, we do not at the moment know any useful ways of controlling this invariant. For example, when  $X$ is  the symmetric square of a curve of large genus, as in Example \ref{Cov.Gon.Exs} (ii), we suspect that $\covgon(X) < \conngon(X)$, but we do not know how to prove this.

\section{Voisin's Invariant}

In this section, we sketch the basic properties of a Chow-theoretic measure of irrationality. This material was suggested to us by Claire Voisin.

Let $X$ be a smooth complex projective variety of dimension $n$. A \textit{one parameter family of Chow-constant $zero$-cycles} on $X$ consists of a family of effective zero cycles $\{ Z_t\}_{t \in T}$ parametrized by a smooth irreducible curve $T$ with the property that all the $Z_t$ are rationally equivalent to a fixed cycle. Voisin's idea is to consider the least degree of such cycles passing through a general point of $X$. 
\begin{definition} We define the \textit{Voisin invariant} $v(X)$ to be the least positive integer $v > 0$ with the property:
\begin{quote} For  any proper algebraic subset $W \subseteq X$, and a general point $x \in X$ not lying on $W$, there exists a one-parameter family $\{ Z_t \}_{t \in T}$ of reduced Chow constant zero cycles of degree $v$ with the properties that
\vskip5pt
\begin{itemize}
\item[(i).]
$x \, \in \, Z_0 \ \text{for some } 0 \in T$;
\vskip 5pt
\item[(ii).] $\textnormal{Supp}(Z_t) \ \text{is disjoint from } W  \text{ for general } t \in T$. 
\end{itemize}
\end{quote}
\end{definition}
\noi It follows from (ii) that $v(X)$ is a birational invariant of $X$.

\begin{example}
One evidently has
\[   v(X) \ \le \ \covgon (X). \]
If $\chow_0(X)$ is trivial, then $v(X)=1$. In particular, if $X$ is Chow-trivial but not uniruled, then $v(X) < \covgon(X)$. 
\end{example}

The following result generalizes Theorem \ref{Cov.Gon.via.BVA.Thm}:
\begin{theorem} \label{Voisin.Invariant.Theorem}
If $K_X$ satisfies $\BVA{p}$, then $v(X) \ge p+2$.
\end{theorem}

\begin{proof}[Proof]
Assume  there is a one-dimensional family of Chow-constant zero-cycles of degree $v$ passing through a general point of $X$. Then by a standard argument they must fit together in an $n$-dimensional family of cycles dominating $X$. More precisely, there exists an $n$-dimensional smooth irreducible variety $S$ admitting a morphism
\[    F : S \lra \Sym^v(X) \]
with the properties that
\begin{itemize}
\item[(i).] The pull-back 
\[   Z_S\ =_{\text{def}} S \, \times_{\Sym^v(X)} \, \big( \, X \times \Sym^{v-1}(X)\, \big ) \]
to $S$ of the universal zero-cycle dominates $X$;
\vskip 5pt
\item[(ii).] $Z_S$ is generically \'etale over $S$ (ie the generic cycle $Z_s$ is reduced).
\end{itemize}
After possibly shrinking $S$, we are free to suppose that $Z_S$ is actually \'etale over $S$.
In this setting, Mumford \cite{Mumford} constructs a trace mapping
\[   \textnormal{Tr}_F : H^0(X, K_X) \lra H^0(S, K_S): \]
for $\eta \in H^0(X, K_X)$ and $s \in S$, $\textnormal{Tr}_F(\eta)$ is determined by the formula
\[   \textnormal{Tr}_F(\eta)(s) \ = \ \sum_{x \in Z_s} \eta(x). \]
Lemma \ref{Technical.Trace.Lemma} below shows that because $S$ is constructed from  Chow-constant one-parameter families, one has
\[   \textnormal{Tr}_F \ = \ 0. \]
But this implies that  the points of the general cycle $Z_s$ satisfy the 
\textit{Cayley-Bacharach property} with respect to $\linser{K_X}$, ie any $n$-form vanishing at all but one of the points of $Z_s$ vanish at the remaining one. (Compare for instance \cite[Proposition 2.3]{BCD}  or \cite[\S 3.2 -- \S 3.4]{GP}.) Therefore $\BVA{v-1}$ fails for $X$, as required.
\end{proof}

\begin{lemma}\label{Technical.Trace.Lemma}
In the setting of the proof of Theorem \ref{Voisin.Invariant.Theorem}, 
\[  \trace_F(\eta) \, = \, 0 \ \ \text{for any } \eta \in H^0(X, K_X).\]
\end{lemma}

\begin{proof}[Sketch of Proof]
For lack of a suitable reference, we sketch the modifications of the arguments from \cite{Mumford} required to verify the assertion. 
Note to begin with that to give a one-parameter family $\{ Z_t \}_{t \in T}$ of Chow-constant degree $v$ zero-cycles amounts to specifying maps
\[  f : T \lra \Sym^v(X) \ \ , \ \ a : T \lra \Sym^w(X) \]
together with a morphism
\[   h : T \times \PP^1 \lra \Sym^{w+v}(X) \]
satisfying
\[   h(t,0) = f(t_0) + a(t) \ \ , \ \ h(t , \infty) = f(t) + a(t), \]
where $t_0 \in T$ is a fixed point. In the setting of the proof of Theorem \ref{Voisin.Invariant.Theorem}, we can therefore suppose (after possibly shrinking and replacing $S$ by an \'etale covering) that we have:
\begin{itemize}
\item[(i).] A smooth surjective morphism 
$ \pi : S \lra B$ of relative dimension one together with a section 
$\sigma: B \lra S$ fixing a base-point on the fibres of $\pi$;
\vskip 5pt
\item[(ii).] A morphism $A : S \lra \Sym^w(X)$;
\vskip 5pt
\item[(iii).] A morphism
\[   H : S \times \PP^1 \lra \Sym^{v+w}(X) \]
satisfying
\[   H(s,0) \, = \, F( \sigma\pi(s)) + A(s)\ \ , \ \ H(s, \infty) \, = \, F(s) + A(s).  \tag{*} \]
\end{itemize}

Now consider a form $\eta \in H^0(X, K_X)$. By the functoriality of Mumford's construction, one has
\[   \trace_{F+A}(\eta) \ = \ \trace_F(\eta) \, + \, \trace_A(\eta). \]
On the other hand, thanks to (*):
\begin{align*}
\trace_H(\eta) \mid S \times \{ \infty \} \ &= \ \trace_{F}(\eta) + \trace_A(\eta) \\
\trace_H (\eta) \mid S \times \{ 0 \} \ &= \ \trace_{F \circ \sigma \circ \pi}(\eta) + \trace_A(\eta)  \, = \, 0 +  \trace_A(\eta),
\end{align*}
the last equality arising from the fact that $F\circ \sigma \circ \pi$ factors through a variety of dimension $n -1$. But since $\PP^1$ carries no canonical forms, 
\[  \trace_H(\eta) \mid S \times \{\lambda \}  \, \in \, H^0(S, K_S) \]
is independent of $\lambda \in \PP^1$. 
Thus $\trace_F(\eta) = 0$, as required.
\end{proof}

\section{Degree of Irrationality of Projective Hypersurfaces}
\label{Deg.Irrat.Section}

In this section we discuss the degree of irrationality and give the proof of Theorem \ref{BCP.Conj.Thm.} from the Introduction. 

We start with some general remarks about the irrationality degree $\irrdeg(X)$ of an irreducible complex projective variety $X$ of dimension $n$. Recall from the Introduction that this is defined to be the least degree of a dominant rational map
\[  f : X \dra \PP^n. \]
Equivalently,  one can characterize $\irrdeg(X)$ as the minimal degree of a field extension
\[    \CC(t_1, \ldots, t_n) \ \subseteq \ \CC(X)  \]
where the $t_i \in \CC(X)$ are algebraically independent rational functions on $X$. We refer to \cite{Y1}, \cite{Y2}, \cite{Y3}, \cite{Y4}, \cite{Bast} for some computations and estimations of $\irrdeg(X)$, especially  in the case of surfaces. 

Given a rational covering $f : X \dra \PP^n$, observe that the inverse images of lines $\ell \subseteq \PP^n$ determine a family of curves  of gonality $\le \deg(f)$ connecting two general points on $X$. This shows that
\begin{equation}\label{Ineq.Invars.In.Section}
\covgon(X) \ \le \ \conngon(X) \ \le \ \irrdeg(X).
\end{equation}
The existence of rationally connected varieties that are not rational  -- as well as many other examples -- illustrates that the gonality invariants can be strictly smaller than $\irrdeg(X)$. However by combining \eqref{Ineq.Invars.In.Section} with Theorem \ref{Cov.Gon.via.BVA.Thm}
 we find:
\begin{corollary} Let $X$ be a smooth projective variety whose canonical bundle $K_X$ satisfies Property $\BVA{p}$ for some $p \ge 0$. Then
\[ \irrdeg(X) \ \ge \ p+2. \hfill\qed\]
\end{corollary}
\noi As above (Examples \ref{Examples.BVA} (v) and \ref{Cov.Gon.Exs} (v)),  equality holds for the desingularization of  a hypersurface  of degree $d$ in $\PP^{n+1}$ with an ordinary $n$-fold point.

\begin{remark}  \label{Direct.Proof.Irr.Bound}
One can give a direct proof of (a strengthening of) the Corollary using  results and methods of  \cite{LP}, \cite{Bast}, \cite{BCD}  and \cite{GP}, involving correspondences with null trace and the Cayley-Bacharach property, along the lines of the proof of Theorem \ref{Voisin.Invariant.Theorem}.  Specifically, consider a dominant  rational map 
\[   f : X \dra Y  \]
 between two smooth projective $n$-folds. We claim:
\begin{equation} \label{Null.Trace.Eqn} \text{  If $K_X$ satisfies $\BVA{p}$ and $H^0(Y,K_Y ) = 0$, then 
$ \deg f \ \ge \ p+2. $} \end{equation}
In fact, given any rational covering $f$ one has a trace map
\[ \textnormal{Tr}_f : H^0(X, K_X) \lra H^0(Y,K_Y) \] 
on canonical forms. As before, for $\eta \in H^0(X, K_X)$ and a general point $y \in Y$,  one can view the value of $\textnormal{Tr}_f(\eta)$ at $y$ as being computed by averaging  the values of $\eta$ over the fibre $f^{-1}(y)$ of $y$. 
It follows (as in the proof of Theorem \ref{Voisin.Invariant.Theorem}) that if $H^0(Y, K_Y) = 0$, then  $f^{-1}(y)$ satisfies the Cayley-Bacharach property with respect to $\linser{K_X}$, ie any $n$-form vanishing on all but one of the points of $f^{-1}(y)$ must vanish on the remaining one. (See for instance \cite[Proposition 2.3]{BCD}  or \cite[\S 3.2 -- \S 3.4]{GP}.) In particular, these points do not impose independent conditions on $H^0(X, K_X)$, and \eqref{Null.Trace.Eqn}  follows. \qed \end{remark}

\begin{remark}
Voisin has pointed out to us that one can also prove  a variant of the statement \eqref{Null.Trace.Eqn} from the previous remark. Specifically,  consider   a smooth polarized projective $n$-fold   $(X, H)$ with the property that the resulting primitive Hodge-structure $H^n(X, \QQ)_{\text{prim}}$ is irreducible: this holds for instance for a very general hypersurface $X \subseteq \PP^{n+1}$ of degree $ > n+2$. Suppose moreover that $K_X$ satisfies property $\BVA{p}$  with $p \ge 1$. If $Y$ is \textit{any} smooth projective variety of dimension $n$, then any rational covering
\[  f : X \dra Y \ \ \text{ with } \deg(f) < p+2 \]
must actually be birational. In fact, assume to the contrary that $f$ is not birational. The mapping $f^* H^0(Y, K_Y) \lra H^0(X, K_X)$, which in any event is injective, must be surjective or zero, else it would give a non-trivial Hodge substructure of $H^n(X, \QQ)_{\text{prim}}$. The former possibility is impossible since $K_X$ satisfies $\BVA{1}$, and therefore it must be the case that $H^0(Y, K_Y) = 0$. Then the previous remark applies. (Compare \cite[Proposition 3.5.2]{GP}.)
\end{remark}

We now turn to the case of a smooth hypersurface
\[   X \ \subseteq\  \PP^{n+1}\]
of dimension $n \ge 2$ and degree $d \ge n+2$. Projection from a point of $X$ shows that in any event
\begin{equation} \label{UpperBound irrat hyps} \irrdeg(X) \ \le \ d-1, \end{equation}
and by \cite[Theorem 1.2]{BCD} (or Corollary \ref{Cov.Gon.Hypsfs}  above) one has the lower bound
\begin{equation} \label{LowerBound.irrat hyps} \irrdeg(X) \ \ge \ d - n. \end{equation}

\begin{example} Interestingly enough,   it can actually happen that $\irrdeg(X) < d-1$. For instance, suppose that $X \subseteq \PP^3$ is a surface containing two disjoint lines 
$  \ell_1 ,\ell_2  \subseteq  X$. Then the line   joining general points $p_1 \in \ell_1, p_2 \in \ell_2$ meets $X$ at $(d-2)$ residual points, and this defines a rational mapping
\[   X\,  \dra \,  \ell_1 \times \ell_2 \, \approx \, \PP^2 \]
of degree $d-2$. There are a  few other examples of a similar flavor, and  it is established in \cite[Theorem 1.3]{BCD} that these are the only surfaces of degree $d \ge 5$ in $\PP^3$ for which $\irrdeg(X) = d-2$. In a similar way, if $X \subseteq \PP^{2k+1}$ contains two disjoint $k$-planes then $\irrdeg(X) \le d-2$, but apparently no examples are known of hypersurfaces of odd dimension $\ge 5$ for which equality fails in \eqref{UpperBound irrat hyps}. (See \cite[4.13, 4.14]{BCD}.) \qed \end{example}

Our main goal is to establish Theorem \ref{BCP.Conj.Thm.} from the Introduction. We will work with the following
\begin{setup} \label{Setup}
Denote by $X \subseteq \PP^{n+1}$ a smooth hypersurface  of degree $d$, and suppose given a rational covering
\[    f : X \dra \PP^n \]
of degree $\delta$. We assume that we are in one of the following two situations:
\begin{itemize}
\item[(a).]  $d \ge 2n + 1$ and $\delta \le d - 2$; or
\vskip 5pt
\item[(b).] $d \ge 2n + 2$ and $\delta \le d-1$. 
\end{itemize}
\end{setup}
\noi  The argument will draw on some  constructions and results of \cite{BCD}. Specifically, the rational mapping $f : X \dra \PP^n$ is given by a correspondence 
 \[    Z \ \subseteq \ X \times \PP^n, \]
and for any $y \in \PP^n$ we can view the fibre $Z_y$ -- which in general consists of $\delta$ distinct points of $X$ -- as a subset of $\PP^{n+1}$. With this notation one has:
\begin{proposition} 
\label{BCDProp} Assume that we are in the situation of \ref{Setup} $($a$)$  or $($b$)$. \begin{itemize}
\item[$($i$)$.] For general $y \in \PP^n$ the fibre $Z_y \subseteq \PP^{n+1}$ lies on  a line \[ \ell_y \subseteq \PP^{n+1}.\]
\item[$($ii$)$.] A general point of $\PP^{n+1}$ lies on exactly one of these lines.
\end{itemize}
\end{proposition}

\noi These facts are established in \cite[Theorem 2.5, Lemma 4.1]{BCD} using the ideas involving correspondences with null trace recalled in Remark \ref{Direct.Proof.Irr.Bound}
 above.

We  now explain the basic geometric idea of the proof. Assume that $f$ is not projection from a point. For general $y \in \PP^n$, write
\[   \ell_y \cdot X \ = \ Z_y \, + \, F_y, \]
where $F_y $ is a zero-cycle of degree $d - \delta$, and recall  that we already know that that $d - \delta \le n$. As $y$ varies over a suitable rational subvariety of $\PP^n$, the $F_y$ (or subcycles thereof) will sweep out  a subvariety $S \subseteq X$ of dimension $s \ge 1$ having covering gonality $e \le d - \delta$. We prove that $e$ and $s$ are related by the inequality
\[   e(n-s) \, \le \, n,  \tag{*} \]
meaning roughly that if  $s$ is small then $e$ cannot be too large. On the other hand, a variant of  the arguments of \cite{Ein}  and \cite{Voisin} shows that if $e < d -2n +s$, then a very general hypersurface of degree $d$ does not contain an irreducible subvariety $S \subseteq X$ of dimension $s > 0$ with $\covgon(S) = e$ (Proposition \ref{No.Curves.Small.Gonality}). This turns out to contradict (*) when $d \ge 2n+2$, and with a little more care a similar argument works when $d  = 2n+1$.

 Moving to details, we start by fixing some further notation. We assume until the end of the proof of Theorem \ref{BCP.Conj.Thm.}   that we are in the situation of  \ref{Setup} (a) or (b), so that Proposition \ref{BCDProp} holds. In classical language, the  lines $\{ \ell_y\}$ form a \textit{congruence} of lines, ie a family of lines parametrized  by an irreducible $n$-dimensional subvariety
 \[   B_0 \ \subseteq \ \GG \, = \, \GG(\PP^1, \PP^{n+1}) \]
 of the Grassmannian of lines in $\PP^{n+1}$. Statement (ii) of the Proposition asserts that the congruence has $\textit{order one}$: if
$W_0 \subseteq    B_0 \times \PP^{n+1} $
 is the restriction to $B_0$ of the tautological point-line correspondence in $\GG \times \PP^{n+1}$, this means that  the projection \[ \mu_0: W_0 \lra \PP^{n+1}\] is birational, and it implies that $B_0$ is rational.\footnote{If one fixes a general hyperplane $H \subseteq \PP^{n+1}$, then almost every point of $H$ lies on a unique line of the congruence,  establishing a birational isomorphism $H\approx B_0$.}    
 
 Replacing $B_0$ by a desingularization $B \lra B_0$, we arrive at the basic diagram:
 \begin{equation} \label{basic.diagram}
 \begin{gathered}
 \xymatrix @C = 2.5pc@R=.55pc {
 X  &X^\pr \ar[l]_{\mu^\pr}\\
 \rotatebox{90}{$\supseteq$}    &\rotatebox{90}{$\supseteq$}\\
 \PP^{n+1} & W \ar[l]^\mu \ar[ddd]_\pi ^{\PP^1-\text{bundle}}\\ \\ \\
 &B \ar[r] &\GG
 }
 \end{gathered}
  \end{equation}
 Here $B$ is a smooth rational $n$-fold mapping birationally to its image in the Grassmannian $\GG$, and $\pi : W \lra B$ is the pull-back to $B$ of the tautological $\PP^1$-bundle on $\GG$. The  mapping $\mu : W \lra \PP^{n+1}$ is birational, and we define $ X^\pr \subseteq W$  to be the proper transform of $X$ in $W$. Thus $X^\pr$ is a reduced and irreducible divisor in $W$ of relative degree $\delta$ over $B$, and  $X^\pr \lra B$ is a generically finite morphism of degree $\delta$ that represents birationally the original mapping $f : X \dra \PP^n$.

Keeping this notation, we now give the:
 \begin{proof}[Proof of Theorem \ref{BCP.Conj.Thm.}] Let 
\[   X^* \ = \ \mu^{*}(X) \ \subseteq \ W \]
be the full pre-image of $X$ in $W$, so that $X^*$ is a (possibly non-reduced) divisor in $W$ of relative degree $d$ over $B$. We can write
$ X^*= X^\pr   +  F$, where $F$ is a divisor of relative degree $d - \delta \ge 1$ over $B$.
Now fix any irreducible component $Y \subseteq F$ that dominates $B$,  and view $Y$  as a reduced irreducible variety of dimension $n$. Thus $Y$ sits in a diagram
\begin{equation}\begin{gathered} \label{YtoXDiagram}
\xymatrix @C = 2.5pc@R=.55pc {
 X  &Y \ar[l]\\
 \rotatebox{90}{$\supseteq$}    &\rotatebox{90}{$\supseteq$}\\
 \PP^{n+1} & W \ar[l]^\mu \ar[r]_\pi &B,}
\end{gathered}
\end{equation}
  and we have
  \begin{equation} 
  \label{e.equation} 0 \ < \ e \, =_{\text{def}} \, \deg( Y \lra B ) \ \le \ d - \delta. \end{equation}

 Put 
 \begin{equation} \label{DefS}S \ =_\text{def}\  \mu(Y)\, \subseteq \, X, \end{equation} and let $s = \dim S$. Suppose first that  $s = 0$, ie that $S$ consists of a single point $p \in X$. This means that all the lines in the congruence pass through $p$, and hence $f$ must be projection from $p$. Therefore we may henceforth assume that 
$s  \ge 1$. 

Note next that $\covgon(S) \le e$. In fact, one can choose a rational subvariety $L \subseteq B$ of dimension $s$ with the property that an irreducible component  $Y^* \subseteq Y$ of the inverse image of $L$ in $Y$ is generically finite over $S$. Since $\deg(Y^* \lra L) \le e$, and since $L$ is rational, we see that $\covgon(Y^*) \le e$. Hence Lemma \ref{Covering.Lemma} applies to show that $\covgon(S) \le e$.  
 
 Now denote by  $K_{W/\PP} = K_{W/\PP^{n+1}}$ the relative canonical bundle of $\mu$ -- ie the ramification divisor of $\mu $ -- and consider a general fibre $\ell = \ell_y$ of $\pi$, ie a general line in our congruence. Recall that by a classical and elementary calculation, one has
 \begin{equation}   \label{Rel.Can.Eqn} (\ell \cdot K_{ W/ \PP}\big) \ = \ n \end{equation} 
(\cite[Lemma 1.1]{Arrondo}). 
 On the other hand, since $\dim \mu(Y) = s$ we claim that  \begin{equation}
 \label{Rec.Can.Ineq}  \ord_Y\big(K_{W/\PP}\big) \, \ge \, n-s. \end{equation}
 This follows from a standard computation, but we sketch a proof   in the Appendix (Corollary \ref{Lower.Bound.Rel.Canon}).  Therefore the contribution from $Y$ to the intersection product in \eqref{Rel.Can.Eqn} is 
$ \ge \, (n-s)e$, so in other words
\[
e(n-s) \ \le \ n . \tag{*}
\]

Now recall that we assume that  $s = \dim S \ge 1$. Then it follows from Proposition \ref{No.Curves.Small.Gonality} below  that
\[  e \ge d -2n +s. \]
Combining this with (*), one finds that
\[  d - 2n + s \ \le \ \frac{n}{n-s}. \] 
 But
 \[  \frac{n}{n-s} - s \ \le \ 1,  \]
 when   $1 \le s \le n-1$, which forces $d \le 2n + 1$. Therefore, if $d \ge 2n + 2$ and $\delta \le d-1$, then $s = 0$, ie $f$ is given by projection from a point.
 
Suppose next that $d = 2n + 1$. Here it remains to rule out the possibility that $s = n-1$ and $e = n = d -n - 1$. In this case, $Y$ is the unique irreducible component of the exceptional locus that dominates $B$, and 
 \begin{equation} \label{2n+1.Equation} \ord_Y(K_{W/\PP})  \ = \ 1. \end{equation} We claim that the mapping $Y \lra S$ is birational to  a $\PP^1$ bundle over $S$. In fact, as in \cite[p. 113]{KollarSmith}, $Y$ is birational to an exceptional divisor in a sequence of blowings-up of $
 \Bl_S(\PP^{n+1})$ along smooth centers dominating $S$, and then \eqref{2n+1.Equation} implies that $Y$ is birational to the dominating exceptional divisor of the blow-up of  $S$ itself. Now let $C \approx \PP^1$ be a generic fibre of $Y \lra S$, and write
\[   \pi^{-1} \big( \pi (C) \big) \ = \ C \cup C^\pr, \]
$C^\pr$ being a curve which is finite over its image $\pi(C)$. 
Thus each of the components of $C^\pr$ has gonality $\le e -1 = d-n-2$. We assert:
\begin{equation} \label{assertion}
\text{Every component of $C^\pr$ maps to a curve in $S$.}
\end{equation} Grant this for now.  One checks moreover that as one varies the fibres of $Y \lra S$, the resulting curves $C^\pr$ generically cover $S$. It follows that 
\[  \covgon(S) \le d - n - 2, \]
which contradicts Proposition \ref{No.Curves.Small.Gonality} and we are done.

It remains to prove \eqref{assertion}. Since $C$ is a fibre of $Y \lra S$, any component of $C^\pr$ that contracts would have to map to a different point of $S$. This means that all  the lines $\ell_z \subseteq \PP^{n+1}$ parametrized by $z \in \pi(C) \subseteq B$ share two points in common, and hence coincide. This in  turn implies that $\pi(C)$ maps to a point in the Grassmanian $\mathbf{G}$. But this is impossible since we may assume that $\pi(C)$ passes through a general point of $B$. 
    \end{proof}
    
    \begin{remark} \textbf{(Fundamental locus of a congruence of order one)} An argument similar to the one just completed  shows that every irreducible component of the fundamental locus of a congruence of order one other than a star of lines has covering gonality $\le n$. \qed\end{remark} 
  
We next show that the  computations of \cite{Ein} and \cite{Voisin} lead to the following statement.  This is essentially the same argument that appears in \cite{F}. 
  \begin{proposition} \label{No.Curves.Small.Gonality}
Let $X \subseteq \PP^{n+1}$ be a very general hypersurface of degree $d \ge 2n$. If $X$ contains an irreducible subvariety $S \subseteq X$ of dimension $s>0$ and  covering gonality  $\covgon(S) = c$,
then 
\[   c \ \ge \ d  - 2n + s. \]
\end{proposition}

 \begin{proof}[Proof of Proposition \ref{No.Curves.Small.Gonality}]
 Let $V =  \HH{0}{\PP^{n+1}}{\OO_\PP(d)}$ be the vector space of all hypersurfaces of degree $d$, which we view as an affine variety. Let  \[ \mathcal{X} \subseteq V \times \PP^{n+1}\] be the universal hypersurface of degree $d$. Denote by
 \[  \pro_1: \mathcal{X} \lra V \ \ , \ \ \pro_2: \mathcal{X} \lra \PP^{n+1} \] the two projections, and write $v = \dim V$.

Suppose now that a very general hypersurface of degree $d$ contains a subvariety of dimension $s$ having covering gonality $c$. Then by a standard argument there exists a commutative diagram:
\[
\xymatrix{\mathcal{S} \ar[r]^f \ar[d]_\pi & \mathcal{X} \ar[d]^{\pro_1}\\
T\ar[r]_\rho & V, }
\]
where $\pi : \mathcal{S} \lra T$ is a family of varieties of dimension $s$ having covering gonality  $c$,  $\rho$ is \'etale, and $f_t : S_t \lra X_{\rho(t)}$ is birational onto its image.    In this setting, Ein and Voisin prove that if $t \in T$ is a general point, then 
\[ \Omega^{v+s}_{\mathcal{S}} \otimes \Big((\pro_2 \circ f)^* \OO_{\PP^{n+1}}(2n +2 -d-s)\Big) \Big| S_t\]
is generically generated by its global sections (cf \cite[Theorem 1.4]{Voisin}), where $S_t = \pi^{-1}(t)$ is the fibre of $\pi$. This implies that
\[  K_{S_t} \ \lin \ (d+s -2n -2) H_{S_t} \, + \, ( \text{effective}), \]
where $H_{S_t}$ is the pull-back of the hyperplane bundle from $\PP^{n+1}$. Thus $K_{S_t}$ satisfies property $\BVA{d +s -  2n -2}$, and Theorem \ref{Cov.Gon.via.BVA.Thm}
 applies to show that $c \ge d + s - 2n$. 
 \end{proof}
 
 \begin{remark} \textbf{(Ran's Theorem).} As a referee of an earlier version of this paper pointed out, the proof of Theorem \ref{BCP.Conj.Thm.} is related to a result of Ran \cite{Ran}. Ran's theorem asserts that if $S \subseteq \PP^{n+1}$ is an irreducible projective variety of dimension $s \le n-1$, and if the union of the $k$-secant lines to the regular locus of $S$ fill $\PP^{n+1}$, then $k \le s+1$.  At least if one knew that it were smooth, the variety $Y$ in diagram \eqref{YtoXDiagram} would determine such a family with $k = e$, and hence $e \le s+1$. Therefore as above Proposition \ref{No.Curves.Small.Gonality} would yield
 \[    s + 1 \, \ge \, e \, \ge \, d-2n +s, \]
 which implies $d \le 2n +1$. So when $d \ge 2n + 2$ we could infer that $s = 0$, and the case $d = 2n+1$ would be treated as before. However the  argument given above not only avoids questions of singularities in applying Ran's theorem, it also leads to a quick new account of that result, as well as a related theorem of Behesti and Eisenbud \cite{BE}. This is presented in the Appendix.
 \end{remark}

\section{Open Problems}

In this section, we propose some  problems concerning this circle of ideas.

A first natural line of investigation is to compute or estimate the various measures of irrationality for other classes of varieties. As a start, suppose that
\[   X \ \subseteq \ \PP^{n+e} \]
is  a smooth complete intersection of hypersurfaces of degrees $2 \le d_1 \le d_2 \le \ldots \le d_e$. Then 
\[  K_X \ = \ \big(\sum d_i - n - e - 1\big)H, \]
and so it follows immediately from Theorem \ref{Cov.Gon.via.BVA.Thm} that 
\[ \covgon(X) \ \ge \ \sum d_i -n -e + 1. \]
But it seems almost certain that one can do much better:
\begin{problem} \label{Complete.Intersection.Problem}
Find bounds on the birational positivity of $K_X$ $($in the sense of Definition \ref{BVA}$)$ and the irrationality invariants of $X$ that are multiplicative in the degrees of the defining equations.
\end{problem}
\noi For example, when $X \subseteq \PP^{e+1}$ is a complete intersection curve, it was established in \cite[Exercise 4.12]{LLS} that \[ \gon(X) \ \ge \ (d_1 -1)\cdot d_2 \cdot \ldots \cdot d_e, \] but this already used some non-trivial vector bundle technology. In higher dimensions one might therefore want to consider first the case that $X$ is a \textit{very general} complete intersection of the stated multidegrees. In the codimension two case, for instance, Stapleton \cite{Stapleton} shows that if $X \subseteq \PP^{n+2}$ is a very general complete intersection of type $(2,d)$ such that $d \geq 2n$, then $$\irrdeg(X)=d.$$

It also seems  natural to consider polarized $K3$ and abelian surfaces of growing degree.
\begin{conjecture} \label{K3.Conj}
Let $(S_d, B_d)$ be a very general polarized $K3$ surface of genus $d$.\footnote{In other words, $B_d$ is an ample line bundle on $S_d$ with $\int c_1(B_d)^2 = 2d -2$.} Then
\[
\limsup _{d\to \infty} \, \irrdeg(S_d) \ = \ \infty,\]
with an analogous statement   for  a very general abelian surface $A_d$ that carries a polarization of type $(1,d)$. \end{conjecture} 
 \vskip -5pt
 \noi  One might have imagined that the irrationality degree grows linearly in $d$, but Stapleton \cite{Stapleton} observes  that
in fact \[ \irrdeg(S_d) \ \le \ \text{(Constant)}\cdot \sqrt{d}, \]
  so at best the growth is sublinear. Recalling that $\covgon(S_d) = \covgon(A_d) = 2$ for every $d$ (Example \ref{Cov.Gon.Exs}), the Conjecture would yield a natural family of examples showing that the covering gonality and the degree of irrationality capture very different phenomena. 

For higher dimensional abelian varieties, the covering gonality already seems very interesting.
\begin{problem}
Let $A^g$ be a very general principally polarized abelian variety of dimension $g$. Is it the case that
\[   \limsup_{g \to \infty}  \big( \covgon(A^g)\big)  \ = \ \infty?\]
\end{problem}
\vskip -5pt
\noi It follows from a theorem of Pirola \cite{Pirola} that in any event $\covgon(A) \ge 3$ when $\dim A \ge 3$. 

It finally seems very appealing to try to say something about these invariants for important moduli spaces that arise in algebraic geometry. For example, Donagi  suggests:
 \begin{problem} \label{Moduli} Find  bounds on the irrationality invariants of the moduli space $\mathcal{M}_g$ parameterizing  curves of genus $g$, or the moduli space $\mathcal{A}_g$ of principally polarized abelian varieties of dimension $g$.\footnote{Since these are birational invariants, they are  well-defined for open varieties.}
 \end{problem}
  \vskip -5pt
 \noi Presumably lower bounds for $\mathcal{M}_g$ would be very difficult to establish, and  may well be out of reach. On the other hand, upper bounds -- which already seem of some interest -- are likely to be much more accessible. For instance, an elementary argument with Hurwitz schemes shows that 
 \[    \covgon( \mathcal{M}_g ) \ \le \  h_{g,g+1},\]
 where $h_{g,g+1}$ denotes the Hurwitz number counting  degree $g+1$ simple coverings of $\PP^1$ with fixed branch points by a curve of genus  $g$. However this integer is huge, so a natural problem is to find substantially better upper bounds. Concretely, this amounts to asking for constructions that realize a general curve of genus $g$ as the fibre of a surface mapping to a curve $C$ of modest gonality. For $\mathcal{A}_g$, on the other hand, some lower bounds might be relatively easy to obtain. 
 
In another direction,    very little has been established so far about  connecting gonality. \begin{problem}
Develop a stock of examples of varieties $X$ for which $\conngon(X)$ can be estimated or computed.
 \end{problem} 
 \vskip -5pt
 \noi Many examples suggest that it is quite common for $\covgon(X) \ll \irrdeg(X)$, but it is not clear at the moment how easy it is for $\covgon(X)$ and $\conngon(X)$ to diverge (although it can certainly happen). Similarly, it would be interesting to have a better understanding of Voisin's invariant $v(X)$. 

It might also be interesting to ask about rational coverings of projective space having non-minimal degree.
\begin{problem} \label{Semigroup?}
Given an irreducible projective variety $X$ of dimension $n$, what can one say about the possible degrees of rational coverings $X \dra \PP^n$? 
\end{problem}
\vskip -5pt
\noi A simple argument suggests that there exists an integer $\delta_0 $ with the property that given any $\delta \ge \delta_0$ one can find a covering $X \dra \PP^n$ of degree $\delta$. Call the least such integer $\delta_0(X)$. How much can $\delta_0(X)$ deviate from $\irrdeg(X)$? What are estimates for $\delta_0(X)$ in the case of very general hypersurfaces $X_d \subseteq \PP^{n+1}$? One can also ask what ``gaps" can appear in the degrees of rational coverings (or the gonalities of covering families). In the case of surfaces, \cite[Theorem 1.3, Corollary 1.7]{LP} give some statements in this direction.

Another natural direction for research is to explore more fully the formal behavior of these invariants. For instance:
 \begin{problem}
 What are the variational properties of the various irrationality measures in families?
 \end{problem}
  \vskip -6pt
 \noi The example of a general surface $X \subseteq \PP^3$ deforming to one containing two lines shows that $\irrdeg(X)$ can decrease under specialization. Are there examples where it increases in a smooth family? This may be more accessible than the corresponding question for rationality itself, which is unknown.
 
 It is well established that questions of rationality become particularly interesting and subtle for varieties defined over fields that are not algebraically closed. This suggests:
  \begin{problem}
  Study measures of irrationality over non-closed fields. 
  \end{problem}
 \vskip -5pt
 \noi For instance if $X$ is defined over a field $k$, one would want to consider the least degree of a rational covering  $X \dra \PP^n$ defined over $k$. Already the case of hypersurfaces seems  potentially interesting. 
 
Finally, in a more speculative vein, a number of new techniques have been introduced to study questions of rationality, such as Koll\'ar's  passage to characteristic $p > 0$ \cite{Kollar},  the Chow-theoretic ideas used by Voisin \cite{Voisin2}, and the combination of these by Totaro \cite{Totaro}. There have also been ideas that rationality could be detected in the geometry of derived categories (eg \cite{Kuzn}). 
It would be very interesting if further ideas along these lines could be used to  say something about measures of irrationality, for example the Voisin invariant $v(X)$ mentioned briefly above.  Similarly, the papers \cite{Lengths} and \cite{HwangTo} show that the gonality of a curve $C$ influences various Riemannian and K\"ahler invariants of varieties associated to $C$. Are there any analogous statements in higher dimensions?

\appendix
\section{The Theorems of Ran and Beheshti--Eisenbud} 
 
The purpose of this appendix is to show how the ideas from the proof of Theorem \ref{BCP.Conj.Thm.} give a simple account  of both Ran's theorem \cite{Ran} and a related result of Beheshti and Eisenbud \cite{BE}.

For smooth varieties, Ran's result is the following:
\begin{theorem}[Ran \cite{Ran}]\label{ran.thm} Let $X \subseteq \PP^N$ be a smooth variety of dimension $n$. Let $\Sec^{n+2}X$ be the variety swept out by all of the $(n+2)$-secant lines of $X$. Then $$\dim(\Sec^{n+2}X) \leq n+1.$$
\end{theorem}
\noindent Ran works more generally with possibly singular varieties, which requires some preliminary definitions:
\begin{definition}
Let $X \subseteq \PP^N$ be a variety of dimension $n$. Let $X^0$ be the smooth locus of $X$. For $k\geq 0$ put:
\begin{gather*}\Sigma_k \, = \, \left\{ \text{lines } \ell \subseteq \PP^N \Big| \textrm{ length}(\ell \cap X) \geq k  \textrm{ and } \text{Supp}(\ell \cap X) \subseteq X^0\right\} , \\
 \Sec^k X^0 \, = \, \bigcup_{\ell \in \Sigma_k} \ell .
 \end{gather*}
\end{definition}

\noindent Ran's stronger statement is:
\begin{theorem}[\cite{Ran}] \label{gen.ran.thm}
Let $X \subseteq \PP^N$ be a $($possibly singular$)$ subvariety of dimension $n$. Then $$\dim(\Sec^{n+2} X ^0) \leq n+1.$$
\end{theorem}

\begin{proof}
Via generic projections, one  reduces to the case when $ X \subseteq \PP^{N}$ with  $N = n+2$.
Assume for a contradiction that $\overline{\Sec^{n+2} X^0 } = \PP^{n+2}$. In this case, we can find a subvariety $$B_0 \subseteq \mathbf{G}(\PP^1, \PP^{N}) = \mathbf{G}$$ of dimension $n+1 = N-1$ parametrizing $(n+2)$-secant lines for $X$ that generically meet $X$ only at smooth points. As in \S \ref{Deg.Irrat.Section}, one then arrives at a diagram:
\begin{equation} \label{Appendix.Diagram}
\begin{gathered}
 \xymatrix @C = 2.5pc@R=.55pc {
 X \subseteq  \PP^{N} & W \ar[l]^(.35)\mu \ar[ddd]_\pi ^{\PP^1-\text{bundle}}\\ \\ \\
 &B \ar[r] &\GG,
 }
 \end{gathered}
\end{equation}
where $B$ is a smooth projective variety of dimension $N - 1= n+1$ mapping birationally to $B_0$, $\pi : W \lra B$ is the pull-back  to $B$ of the tautological $\PP^1$ bundle on $\GG$, and $\mu$ is surjective and generically finite.  
Denoting by $\II_X \subseteq \OO_{\PP^{N}}$ the ideal sheaf of $X$, write
\[\II_X \cdot \OO_W \ = \ \II_{X^\pr} \cdot \OO_W(-D), \]
where $X' \subseteq W$ is a closed subscheme of codimension at least 2, and $D$ is an effective  divisor on $W$.
We may  decompose $D$ as a sum $D = E +  E',$ where $E$ and $E'$ are effective divisors such that every component of $E$ dominates $B$ and $\codim(\pi(E')) \geq 1$.
Now set $$E = \sum a_i E_i,$$ where the $E_i$ are the distinct irreducible components of $E$. Let $\ell$ be a generic fibre of $\pi$. Our assumption implies that $$(E  \cdot \ell) \ \geq \ N \, = \, n+2.$$

Let $K_{W/\PP^N} = K_{W/\PP}$ be the ramification divisor of $\mu$. By the adjunction formula,  $$\big( K_{W/\PP} \cdot \ell\big)  \, = \, N-1 \, = \, n + 1$$
(\cite[Lemma 1.1]{Arrondo}). Note that that $\mu(E_i) \cap X^0 \neq \emptyset$ for each $i$. In fact, each $E_i$ meets a general fibre $\ell$ of $\pi$, while by assumption $\ell$ meets $X$ only at smooth points. It then follows from Lemma \ref{Rel.Canon.Lemma} below that $$\ord_{E_i} (K_{W/\PP}) \, \geq \, 2a_i - 1 \, \geq \,  a_i  \, = \, \ord_{E_i}(E).$$
This implies that $$N-1 \, = \, \big( K_{W/\PP} \cdot \ell\big) \, \geq \,\big( E \cdot \ell \big) \, \geq \, N,$$  a contradiction.
\end{proof}


In \cite{BE}, Beheshti and Eisenbud give an improvement of Ran's theorem. A similar argument also yields  a simplified proof of their result. 

\begin{theorem}[\cite{BE}, Theorem 1.5]
Let $X \subseteq \PP^N$ be a smooth irreducible variety of dimension $n$. If $k \geq 2,$ then $$\dim(\Sec^kX) \leq \frac{nk}{k-1} + 1.$$
\end{theorem}
\noi More generally, for possibly singular $X$ the same inequality holds for 
$\Sec^k X^0$.
\begin{proof}[Sketch of Proof]
As observed by Gruson and Peskine \cite[p. 554]{GruPesh}, it is equivalent to show that 
$$\dim(\Sec^k X^0) \leq n+s,$$ 
for any positive integer $s$   such that $1 \leq s \leq n+1,$  and $k > (n/s) + 1$. The argument then closely parallels the previous proof. In brief, by a generic projection  we can suppose for a contradiction that  $N = n + s + 1$, and that $\overline{ \Sec^kX^0} = \PP^N$. We then arrive at an analogue of diagram \eqref{Appendix.Diagram} with $\dim B = n + s = N -1.$  Defining $\ell$ and  $E = \sum a_i E_i$ as in the proof of Theorem \ref{gen.ran.thm}, one has
\[   E \cdot \ell \ \ge \ k \ \ , \ \ K_{W/\PP} \cdot \ell \, = \, N -1 \, = \, n +s.\]
On the other hand,
\[   \ord_{E_i}(K_{W/\PP}) \geq (s+1)a_i - 1 \geq s \cdot a_i  \]
thanks to Lemma \ref{Rel.Canon.Lemma}. But then
\[ n + s \, = \, N-1 \, = \, K_{W/\PP} \cdot \ell \, = \, \sum \Big(\ord_{E_i}(K_{W/\PP}) E_i \Big)\cdot \ell \, \geq \,s \cdot ( E  \cdot \ell) \, \geq \, s\cdot k.\]
But this implies that $k \le (n/s) + 1$, a contradiction.
\end{proof}

Finally, we spell out for the convenience of the reader the inequalities we have drawn on concerning the relative canonical divisor of a generically finite and surjective morphism between smooth varieties. 
\begin{lemma} \label{Rel.Canon.Lemma}
Consider a generically finite surjective morphism \[ \mu : W \lra P\] between smooth projective varieties of dimension $m$. Let $X\subseteq P$ be an  irreducible subvariety of codimension $c$ and regular locus $X^0 = X_{\text{reg}}$, and let $F \subseteq W$ be a prime divisor on $W$.  Assume that 
\[    \mu(F) \ \subseteq \ X \ \  \text{ and } \ \ \mu(F)\,  \cap \, X^0 \, \ne \, \varnothing,\]
and denote by $a = \ord_F( \II_X \cdot \OO_W)$  the order of vanishing along $F$ of the pull-back to $W$ of the ideal sheaf of $X$.
Then 
\[   \ord_F \big( K_{W/P} \big) \ \ge \ c \cdot a - 1. \]
\end{lemma}

In particular, taking $X = \mu(F)$ this yields:
\begin{corollary} \label{Lower.Bound.Rel.Canon}
Given $\mu : W \lra P$ as above, let $F \subseteq W$ be a prime divisor on $W$. Suppose that $\dim \mu(F) = s$. Then 
\[  \ord_F \big(K_{W/P}\big) \ \ge \ m - 1 -s.  \ \ \qed \]
\end{corollary}

\begin{proof}[Proof of Lemma \ref{Rel.Canon.Lemma}] We argue as in \cite[Lemma 9.2.19]{PAG}. Choose a general point $y \in F$  and let $x = \mu(y) \in X$, which we suppose to be a smooth point of $X$. We can then choose local analytic coordinates $z_1, \ldots, z_m$ on $W$ centered at $y$ and $u_1, \ldots, u_m$ on $P$ centered at $x$ such that locally
\[     F \, = \, \{ z_1 = 0 \} \ \ , \ \ X \, = \,\{ u_1 = \ldots = u_c  = 0 \}. \]
There exist functions $b_i \in \CC\{z\} $ such that
\[  \mu^* (u_i) \, = \, z_1^{a_i} \cdot b_i, 
\]
where $a_i \ge 0$ for each $i$ and $a_1, \ldots , a_c \ge a$.
Then $\mu^* (du_i) = a_i z_1^{a_i -1} b_i  dz_1 + z_1^{a_i} db_i $, and hence
\[   \mu^* \big( du_1 \wedge \ldots \wedge du_m \big ) \ = \ z_1^{\, (\sum a_i) - 1} \cdot g \cdot dz_1 \wedge \ldots \wedge dz_m \]
for some $g \in  \CC\{z\}$. The assertion follows. 
\end{proof}

 %
 %
 %
 %

 \end{document}